\documentclass[runningheads]{llncs}
\usepackage{amssymb}
\usepackage{amsmath}
\usepackage{graphicx} 
\newcommand\bigsubset[1][1.19]{%
 \mathrel{\vcenter{\hbox{\scalebox{#1}{$\subset$}}}}}
\usepackage{booktabs}

\usepackage[T1]{fontenc}
\usepackage{graphicx}
\begin{document}
\title{Chance Constrained Optimization with Complex Variables}
%
%
\author{Raneem Madani\inst{1} \and
Abdel Lisser\inst{1}}
\authorrunning{R. Madani, A. Lisser}
%
\institute{Université Paris-Saclay, CNRS, CentraleSupélec, Laboratoire des Signaux et Systèmes (L2S), 91190, Gif-sur-Yvette, France.\\
\email{\{raneem.madani,abdel.lisser\}@centralesupelec.fr}}
\maketitle              
\begin{abstract}
Optimization problems involving complex variables when solved, are typically transformed into real variables, often at the expense of convergence rate and interpretability. This paper introduces a novel formalism for a prominent problem in stochastic optimization involving complex random variables, termed the Complex Chance-Constrained Problem (CCCP). The study specifically examines the linear CCCP under complex normal distributions for two scenarios: one with individual probabilistic constraints and the other with joint probabilistic constraints. For the individual case, the core methodology reformulates the CCCP into a deterministic Second-Order Cone Programming (SOCP) problem, ensuring equivalence to the original CCCP. For the joint case, an approximation is achieved by deriving suitable upper and lower bounds, which also leads to a SOCP formulation. Finally, numerical experiments on a signal processing application—specifically, the Minimum Variance Beamforming problem with mismatch using MVDR—demonstrate that the proposed formalism outperforms existing approaches in the literature. A comparative analysis between the joint and individual CCCP cases is also included.

\end{abstract}
\section{Introduction}
Mathematical optimization is the systematic process of maximizing or minimizing an objective while adhering to defined constraints. However, many optimization problems are complicated by uncertainties in data or models. This leads to a critical subfield known as Optimization Under Uncertainty, or Stochastic Optimization, it was first introduced by George Dantzig in the 1950s \cite{dantzig1954solution}. This field is widely studied because real-world problems often involve unpredictable factors that cannot be fully anticipated. Addressing these uncertainties is essential to develop solutions that are both reliable and effective.

One prominent framework for handling uncertainty in optimization is chance-constrained programming (CCP), introduced by Abraham Charnes and William W. Cooper in 1959 \cite{charnes1959chance}. CCP incorporates uncertainty directly into the constraints by representing them as probabilistic conditions. In this approach, constraints must be satisfied with a predefined confidence level, ensuring a balance between feasibility and risk. The importance of addressing CCP problems, particularly when the decision variable is complex, is increasingly evident due to their applications in various real-world domains, such as signal processing \cite{ma2012chance} (refer to Section \ref{sec: Section 8}). Notably, tackling these problems requires alternative approaches that leverage the complex nature of the variables, the distribution of stochastic variables, and the relation matrix, which depends on both \( z \) and its conjugate \( \Bar{z} \). These approaches differ significantly from methods that convert complex variables into real ones using $\mathbb{CR}$-calculus (Wirtinger Calculus \cite{wirtinger1927formalen}\cite{remmert1991theory}). The latter approach loses key exploitable properties \cite{goodman1963statistical}\cite{picinbono1996second}.

Another critical reason for directly handling complex variables lies in the computational implications of conversion. Transforming complex variables into real ones doubles the problem's dimensionality, potentially leading to increased computational time as in table \ref{tab:table1}, slower convergence rates, issues related to the curse of dimensionality \cite{bellman1966dynamic}, and reduced algorithmic accuracy, as seen in \cite{gidi2023stochastic}, the authors proved that solving optimization problems in the complex domain yields better accuracy compared to converting the problem to the real domain particularly when applied to quantum computing problems. 
\begin{table}[ht]
    \centering
    \caption{Average objective values and computation times for minimizing $\|\sqrt{A}\; z\|$ using Gradient Descent (GD) and SPSA \cite{gidi2023stochastic} over 50,000 iterations in 10D and 20D complex and real spaces. Complex algorithms ($GD_C$, $SPSA_C$) achieve comparable accuracy to real-space conversions ($GD_R$, $SPSA_R$) while significantly reducing computation time.}
    \begin{tabular}{|c|cccc|cccc|}
    \midrule
       size  & $GD_C$ & $GD_R$ & $GD_C$ & $GD_R$ & $SPSA_C$ & $SPSA_R$ & $GD_C$ & $GD_R$  \\
       & obj & obj & time & time & obj & obj & time & time\\
       \midrule
       10 & 0.0 & 0.1 & 20.3 & 40.1 & 0.7 & 0.5 & 33.7 & 63.1\\
       20 & 0.1 & 0.1 & 49.3 & 180.6 & 1.6 & 1.4 & 76.8 & 275.9\\
       \midrule
    \end{tabular}
    \label{tab:table1}
\end{table}
The motivation of our work stems from the need to address challenges in quantum optimization,
where hybrid classical-quantum optimization techniques have emerged as a powerful approach for solving problems on today's noisy intermediate-scale quantum (NISQ) computers. These methods combine the strengths of classical and quantum computing by running optimization algorithms on a classical computer, guided by objective function values derived from a quantum processor. For instance, in the Quantum Approximate Optimization Algorithm (QAOA) \cite{farhi2014quantum}, the expected cost Hamiltonian matrix is optimized to determine the variational parameters. Converting this problem into its real counterpart can lead to the loss of its physical meaning, the degradation of key characteristics of the unitary operator, and an unnecessary increase in the problem's dimensionality. Uncertainties in quantum systems, such as time-varying noises \cite{knill2005quantum}, inhomogeneous quantum ensembles and uncertainties in the Hamiltonian \cite{young2013adiabatic}\cite{fei2024binary}, necessitate optimization frameworks that can handle unpredictable dynamics while maintaining the quality of control policies. CCP with complex decision variables offers a robust approach to model and manage these uncertainties effectively, ensuring high-quality and reliable solutions in quantum optimization. 

Reformulation techniques are essential for managing chance constraints in optimization by transforming them into more tractable forms for analysis and solution. For individual linear chance constraints, an equivalent reformulation as second-order cone programming (SOCP) is possible under certain conditions \cite{prekopa2013stochastic}\cite{henrion2008convexity}. For linear joint chance constraints with continuous random variables, in \cite{cheng2015chance}, a full proof is provided using copula theory to convert the problem into a deterministic one. Then, various convex approximations have been proposed, such as \cite{cheng2012second}, Cheng and Lisser derived lower and upper bounds using Taylor series and piecewise methods. 

This paper introduces a novel approach for addressing optimization problems involving complex variables under uncertainty, referred to as Complex-Chance-Constraint Programming (CCCP). To the best of our knowledge, the general form of CCP with complex random variables has not been previously formulated or tackled. Our work focuses on the linear case, where both the objective and constraints are linear, and the random variables follow complex normal distributions. We address two key problem classes: The first class is the individual CCCP, where we transform stochastic problems into equivalent SOCPs by exploiting properties of the mean and variance. The second class is the joint CCCP with dependent matrix rows. For this, we convert the joint CCCP into a deterministic problem by employing copulas to handle dependencies among random vectors. As the resulting problem is biconvex, we approximate the solution by deriving lower bounds using a Taylor series approximation and upper bounds using the piecewise approach. The results have since been extended to encompass nonlinear scenarios with respect to decision variables while remaining linear in relation to random variables. To demonstrate their practical utility, we apply these findings to a robust beamforming problem, specifically addressing the MDVR (Minimum variance distortionless response) case with a mismatch.

The paper is organized as follows: Section \ref{sec: Section 2} presents the preliminaries and problem formulation. Section \ref{sec: Section 3} introduces the CCCP problem. Section \ref{sec: Section 4} discusses the normal distribution of complex affine functions. Section \ref{sec: Section 5} presents the Individual-CCCP. Section \ref{sec: Section 6} delves into the Joint-CCCP, detailing the problem and providing upper and lower approximations. Section \ref{sec: Section 8} presents applications from signal processing (MVDR with mismatch), and simulation experiments comparing our approach with approaches from the literature and also comparisons between Individual and Joint CCCP. Finally, Section \ref{sec: Section 9} concludes it all.
\section{Preliminaries}\label{sec: Section 2}
The set of complex vectors is denoted by $\mathbb{C}^n$. The complex number can be written as $z=(x,y)=x+iy$, where $x=\Re(z)$ is the real part of $z$ and $y=\Im(z)$ is the image part of $z$ \cite{ahlfors1979complex}. The complex conjugate of $z$ is defined as $\bar{z}=(x,-y)=x-iy$. The transpose and the complex conjugate transpose of $z$ are denoted by $z^T$ and $z^H$ respectively. The Euclidean norm of $z$ is equal to $\|z\|=\sqrt{zz^H}=\sqrt{x^2+y^2}$. For a complex random vector $z\in\mathbb{C}^n$ \cite{wooding1956multivariate}, let $\mu_z, \Gamma_z$ and $C_z$ be the mean, covariance matrix, and relation matrix, respectively, then the complex normal distribution is given by $z\sim\mathcal{N}_c(\mu_z, \Gamma_z, C_z)$.
\begin{definition}\label{def 1}
     The mean, covariance, and relation matrix  of the random vector $z$ are given by:
\begin{align}
    &\mu_z = \mathbb{E}[z] = \mathbb{E}[x]+ i\mathbb{E}[y]\\
    &Cov(z,z)=\Gamma_z = \Gamma_x + \Gamma_y +i(\Gamma_{yx}-\Gamma_{xy})\\
    &Cov(z,\bar{z})=C_z = \Gamma_x - \Gamma_y + i(\Gamma_{yx} + \Gamma_{xy}) 
\end{align}
\end{definition}
\textbf{Properties:}
\begin{enumerate}
    \item The covariance matrix is Hermitian, and the relation matrix is symmetric. \label{prop: 1}
    \item $Cov(z,w)=\overline{Cov(w,z)}$. \label{prop: 2}
    \item $ \forall\alpha\in\mathbb{C}, \; Cov(\alpha z,w)=\alpha Cov(z,w)$, and $Cov(z,\alpha w) = \bar{\alpha}Cov(z,w)$. \label{prop: 3}
    \item $Cov(\sum_i z_i,\sum_j w_j) =\sum_{i,j} Cov(z_i,w_j)$. \label{prop: 4}
    \item The covariance matrix $\Gamma_z$ is positive semidefinite: $a^H\Gamma_za\geq0, \forall a\in\mathbb{C}^n.$\label{prop: 5}
    \item if $z\sim\mathcal{N}_c(\mu_z, \Gamma_z,C_z)$ is random variable, and if $A\in\mathbb{C}^{m\times n}$ and $b\in\mathbb{C}^n$, then $Az+b\sim \mathcal{N}_c(A\mu_z+b, A\Gamma_zA^H, AC_zA^T).$\label{prop: 6}
\end{enumerate}


\section{Complex Chance Constraint Problem (CCCP)}\label{sec: Section 4}\label{sec: Section 3}
The complex chance-constraint problem is given as follows:
\begin{align*}
       \text{(CCCP)  ~} &\min f(z,\xi)\\
        &s.t.~ g(z,\xi) = 0\\
        &\mathbb{P}[h(z,\xi)\leq 0]\geq p,
\end{align*}\label{CCCP}
where $f:\mathbb{C}^n\mapsto\mathbb{R}$, $g:\mathbb{C}^n\mapsto\mathbb{C}^k$, and $h:\mathbb{C}^n\mapsto\mathbb{R}^m$. The deterministic decision variables are given by the complex vector $z\in\mathbb{C}^n$, while $\xi$ is the complex random vector containing all uncertainties in the form of complex random variables. We have two classes, first one is the individual CCCP which requires that each constraint be satisfied independently with a certain probability, and then $p=[p_1,\cdots,p_m], p_i\in[0,1];\forall i$ is the probability vector. The other class is the joint CCCP which requires that all inequality constraints in the set be satisfied simultaneously with a certain joint probability, and then $p\in[0,1]$. In this paper, we assume that the feasible set is not empty and $p\geq0.5$.

\section{Normal Distribution of Complex Affine Functions}\label{sec: Section 4}
In this section, we introduce two lemmas that play a crucial role in this paper. We discuss the distribution of the real part of affine functions with complex random variables, which are given by $\Re(c^Hz)$ and $\Re(Az-b)$.
\begin{lemma}\label{lemma 1}
    If $c_j;j=1,\cdots,n$ are a random variables with $c_j\sim\mathcal{N}_c(\mu_{c_j}, \Gamma_{c_j}, C_{c_j})$, if the real and image parts of $c_j$ are independent, then $\Re(c^Hz)$ is real normally distributed with 
\begin{align}
    \Re(c^Hz)\sim\mathcal{N} \left(\Re(\mu_c^Hz),\frac{1}{2}(\Re(z^H\Gamma_cz+z^HC_c\bar{z})) \right)
\end{align}
Furthermore, the variance is a quadratic function.
\end{lemma}
\begin{proof}
Using Definition (\ref{def 1}) to find the mean and the covariance matrix, the mean of $F$ is given by:
\begin{align}
&\mathbb{E}[\Re(c^Hz)]=\Re(\mu_c^Hz)
\end{align}
Using property (\ref{prop: 4}) the variance is given by:
\begin{align}
Var&\left(\frac{1}{2} \sum_{j=1}^n \bar{c}_jz_j+\bar{z}_jc_j\right)=\frac{1}{4}(Var(\sum_{j=1}^n \bar{c}_jz_j)+Var(\sum_{j=1}^n \bar{z}_jc_j)\\
&+Cov(\sum_{j=1}^n\bar{c}_jz_j,\sum_{j=1}^n\bar{z}_jc_j)+Cov(\sum_{j=1}^n\bar{z}_jc_j,\sum_{j=1}^n \bar{c}_jz_j))
\end{align}
Decomposing the term by applying the properties of variance, and
using Definition (\ref{def 1}) the variance will be equal to: 
\begin{equation} 
\begin{aligned}
        \frac{1}{2}\Re\left(z^H\Gamma_cz + z^TC_cz\right)=x^T\Gamma_{c^r}x + x^T(\Gamma_{c^rc^i}+\Gamma_{c^ic^r})y + y^T\Gamma_{c^i}y \label{12}
\end{aligned}
\end{equation}
To guarantee the positive semidefinitness of the variance, $\Gamma_{c^rc^i}$ must be zero.
\end{proof}
\begin{lemma}\label{lemma 2}
    Let $a_{ij}\sim \mathcal{N}_c(\mu_{a_{ij}},\Gamma_{a_{ij}}, C_{a_{ij}})\quad j=1,\cdots, n,$ $ i=1,\cdots,m$, and $b_i\sim \mathcal{N}_c(\mu_{b_i},\Gamma_{b_i},C_{b_i}),$ $ i=1,\cdots, m$, if the real and imaginary parts of $A_i=(a_{i1},\cdots a_{in})$ are independent, then $\Re(A_iz-b_i)$ are real normally distribution, with 
\begin{align}
    \Re(A_iz-b_i)\sim\mathcal{N}\left(\Re(\mu_{A_i}z)-\mu_{b_i}, \frac{1}{2}\Re\left(z^H\Gamma_{A_i}z+z^TC_{A_i}z\right)+\sigma_{b_i}\right),~i=1,\cdots, m
\end{align}
Furthermore, the covariance matrix of $\Re(A_iz-b_i)$ is a quadratic function.
\end{lemma} 
\begin{proof}
Similar to Lemma (\ref{lemma 1}) proof.
\end{proof}
From this point forward, we assume that the real and image parts of $c$ are independent, the real and image parts of $A_i,\quad i=1,\cdots,m$ are independent, and $b_i, \; A_i \; \forall i$ are independent, and we consider Linear CCCP.

\section{Individual-CCCP}\label{sec: Section 5}
Let $c_j, a_{ij}$ and $b_i$ be complex random variables with known probability distributions, then:
\begin{align*}
        (P1)~~\min &\sum_{j=1}^n \Re(\bar{c}_jz_j)\\
         \text{s.t. } &\mathbb{P}\left[\sum_{j=1}^n \Re(a_{ij}z_j-b_i)\leq 0\right]\geq p_i, \quad \forall i\\
        &\Re(z)\geq 0,\quad \Im(z)\geq 0\label{P1}
    \end{align*}
\subsection{Individual CCCP as SOCP}\label{sec: 4.2}
This section aims to convert the stochastic problem (P1) to a deterministic problem, which is a complex second-order problem. Using Lemma (\ref{lemma 1}), the new deterministic objective function minimization can be formulated as follows:
\begin{align}
q_1\Re(\mu_c^Hz) + q_2\sqrt{\frac{1}{2}(z^H\Gamma_cz+\Re(z^HC_c\bar{z}))}
\end{align}
Where $q_1$ and $q_2$ are nonnegative weights for the mean and variance. Let
\begin{align}
    h_i = \sum_{j=1}^n \Re(\bar{a}_{ij}z_j)-b_i=\sum_{j=1}^n\frac{1}{2}(a_{ij}z_j+\overline{a_{ij}z_j})-b_i,~\forall i
\end{align}
Thus the constraint of the problem is:
\begin{align}
    \mathbb{P}\left[(h_i-\mu_{h_i})Cov(h_i)^{-\frac{1}{2}}\leq -\mu_{h_i}Cov(h_i)^{-\frac{1}{2}}\right]\geq p_i,\quad \forall i
\end{align}
Let $\Phi(x)$ represent the CDF of the standard normal distribution evaluated at $x$. The constraint is stated as:
\begin{align}
    &\Phi\left(-\mu_{h_i}Cov(h_i)^{-\frac{1}{2}}\right)\geq p_i\iff \mu_{h_i}+\Phi^{-1}(p_i)\sqrt{Cov(h_i)}\leq 0
\end{align}
Then, the deterministic equivalent of problem \ref{P1} can be written as: 
\begin{align*}
   (P2)~~     \min & ~~q_1\Re(\mu_c^Hz)+q_2\sqrt{\frac{1}{2}\Re(z^H\Gamma_cz+z^HC_c\bar{z})}\\
        \text{s.t. }& \Re(\mu_{A_i}z)-\mu_{b_i}+ \Phi^{-1}(p_i)\sqrt{\frac{1}{2}\Re(z^H\Gamma_{A_i}z+z^TC_{A_i}z)+\sigma_{b_i}}\leq0,\quad \forall i\\
        &\Re(z)\geq0,\quad \Im(z)\geq 0
\end{align*}
\begin{theorem} Problem (P2) is a convex problem, furthermore, it's a second-order cone problem.
\end{theorem}
\begin{proof}
Let $s =
\begin{bmatrix}
    \Re(z) & \Im(z) & 1
\end{bmatrix}^T$,
$K_0 =
\begin{bmatrix}
    \Re(\mu_c) & \Im(\mu_c) & 0
\end{bmatrix}^T,$ $ K_{1}=\begin{bmatrix}
    \Gamma_{\Re(c)} & O^{n\times n} & O^{n\times 1}\\
     O^{n\times n} & \Gamma_{\Im(c)} & O^{n\times 1} \\
     O^{1\times n} & O^{1\times n} & 0
\end{bmatrix}$ $ K_{2_i}=\begin{bmatrix}
    \Gamma_{\Re(A_i)} & O^{n\times n} & O^{n\times 1}\\
    O^{n\times n} & \Gamma_{\Im(A_i)} & O^{n\times 1}\\
    O^{1\times n} & O^{1\times n} & \sigma_{b_i}
\end{bmatrix}$, $ K_{3_i} = \begin{bmatrix}
    -\Re(\mu_{A_i}) & - \Im(\mu_{A_i}) & \mu_{b_i}
\end{bmatrix}^T$.\\
Using Lemma (\ref{lemma 2}), the stochastic linear programming problem can be stated as an equivalent deterministic nonlinear programming problem:
    \begin{align*}
       (\text{P3})~~ &\min q_1 s^TK_0 + q_2\|\sqrt{K_{1}}^{T}s\|\\
        &\text{s.t.}~~\Phi^{-1}(p_i)\|\sqrt{K_{2_i}}^{T}s\|\leq s^TK_{3_i}\quad,\forall i\\
        &\quad~~~s\geq 0
    \end{align*}
$K_1$ and $K_{2_i}$ are positive semidefinite matrices, then (P2) gives the SOCP.
\end{proof}

\section{Joint-CCCP}\label{sec: Section 6}
Linear CCCP uses the theory of copulas to represent row dependence. 
    \begin{align*}
       \text{(P4)}~~ \min &\sum_{j=1}^n \Re(\bar{c}_jz_j)\\
         \text{s.t. } &\mathbb{P}\left[\Re(Az-b)\leq 0\right]\geq p, \\
        &\Re(z)\geq 0,\quad \Im(z)\geq 0
    \end{align*}
where $c, A,$ and $b$ are complex normal distributed.
\begin{theorem}
    Problem (P4) is equivalent to the following \emph{bi-convex} problem:
    \begin{align*}
        \text{(P5)}~~\min & ~~q_1\Re(\mu_c^Hz)+q_2\sqrt{\frac{1}{2}\Re(z^H\Gamma_cz+z^HC_c\bar{z})}\\
        \text{s.t. }& \Re(\mu_{A_i}z)-\mu_{b_i} + \Phi^{-1}(p^{y_i^{\frac{1}{\theta}}})\sqrt{\frac{1}{2}\Re(z^H\Gamma_{A_i}z+z^TC_{A_i}z)+\sigma_{b_i}}\leq0,\quad \forall i\\
        & \sum_{i=1}^m y_i=1,\quad\Re(z)\geq0,\quad \Im(z)\geq 0,\quad y\geq 0
    \end{align*}
\end{theorem}
\begin{proof}
    The proof has been done in \cite{cheng2015chance}.
\end{proof}
\subsection{Lower and Upper Approximations of Problem (P5)}\label{sec: 5.2}
To find a SOCP approximation of (P5), we approximate the quantile function by deriving both an upper and a feasible lower bound. 
\subsubsection{Lower Bound Approximation of (P5)}
To formulate the deterministic problem with a lower bound solution to (P5), we find an approximation of quantile function $\Phi_l = \Phi^{-1}(p^{y_i^{1/\theta}})$, using Taylor series expansion around $N$ tangent points $r_l, \; l=1,\cdots,N$,  $r_l \in (0,1]$ with $r_1<r_2<\cdots<r_N$. The approximation function is given by:\\
$\hat{\Phi}_{1l} = \Phi^{-1}(p^{r_l^{1/\theta}}) + \Phi^{-1'}(p^{r_l^{1/\theta}})p^{r_l^{1/\theta}}\ln(p)r_l^{1/\theta-1}\frac{1}{\theta}(y-r_l)$
$=a_{1l}+b_{1l}y$, \\ where $b_{1l}=\Phi^{-1'}(p^{r_l^{1/\theta}})p^{r_l^{1/\theta}}\ln(p)r_l^{1/\theta-1}\frac{1}{\theta}$ and $a_{1l}=\Phi^{-1}(p^{r_l^{1/\theta}})-b_{1l}r_l$. \\
$\hat{\Phi}_1=\displaystyle\max_{l=1,\cdots,N}\hat{\Phi}_{1l}$ is the lower approximation function. 
\begin{theorem}
    \label{thm: 5.1}
Let $\hat{r_i}=(\hat{r}_{i1},\cdots, \hat{r}_{in})$. Together with the approximation of $\Phi^{-1}(p^{y_i^{1/\theta}})$, we have the following approximation of problem (P5), correspondingly:
    \begin{align*}
        (\text{P6})~~\min & ~~q_1\Re(\mu_c^Hz)+q_2\sqrt{\frac{1}{2}\Re(z^H\Gamma_cz+z^HC_c\bar{z})}\\
        \text{s.t. }& \Re(\mu_{A_i}z) + \sqrt{\frac{1}{2}\Re(r_i^H\Gamma_{A_i}r_i+r_i^TC_{A_i}r_i)+\sigma_{b_i}}-\mu_{b_i}\leq0,\forall i\\
        & \Re(r_{ij}) \geq \Re(a_{1l}z_j + b_{1l} m_{ij}),\quad\forall i,j,l\\
        & \Im(r_{ij}) \geq \Im(a_{1l}z_j + b_{1l} m_{ij}),\quad\forall i,j,l\\
        & \sum_{i=1}^m m_{ij}=z_j,\forall j,\quad \Re(z)\geq0,\quad \Im(z)\geq 0,\quad m\geq 0
        \end{align*}
Moreover, the optimal value of this approximation is a lower bound of (P5).
\end{theorem}
\begin{proof}
First, we know that the quantile function $\Phi^{-1}(p^{y^{1/\theta}})$ is convex for all $p\geq 0.5$. Therefore, for any tangent point $r_l, \; l=1,\cdots, N$, 
\begin{align}
    \Phi^{-1}(p^{y^{1/\theta}})\geq \hat{\Phi}_{1} = \max_{l=1}^N \{\hat{\Phi}_{1l}\}.
\end{align}
Since $\sqrt{\frac{1}{2}\Re(z^H\Gamma_{A_i}z+z^TC_{A_i}z)+\sigma_{b_i}}\geq 0$, then $\forall i$ we have:
\begin{align}
    \left\{z:\Re(\mu_{A_i}z) + \Phi^{-1}(p^{y_i^{\frac{1}{\theta}}})\sqrt{\frac{1}{2}\Re(z^H\Gamma_{A_i}z+z^TC_{A_i}z)+\sigma_{b_i}} \leq\mu_{b_i}
    \right\}\\
    \bigsubset [1.7] \left\{z:\Re(\mu_{A_i}z) + \hat{\Phi}_{1}\sqrt{\frac{1}{2}\Re(z^H\Gamma_{A_i}z+z^TC_{A_i}z)+\sigma_{b_i}} \leq \mu_{b_i}\right\}.
\end{align}
\end{proof}
\subsubsection{Upper Bound - Piecewise Linear Approximation}
To construct an upper approximation solution to (P5), we select $N$ interpolation points $r_l,l=1,\cdots, N$ from the interval $(0,1]$, with $r_1<r_2<\cdots<r_N$, and denote $\Phi^{-1}(p^{r_l^{1/\theta}})$ by $\Phi_l$. Let $\hat{\Phi}_{2l}$ be the corresponding piecewise linear approximation of $\Phi^{-1}(p^{y^{1/\theta}})$. We have:
\begin{align}
    \hat{\Phi}_{2l}=\Phi_l + \frac{y-r_l}{r_{l+1}-r_l}(\Phi_{l+1}-\Phi_l)=a_{2l}+b_{2l}y
\end{align}
where $a_{2l}=\cfrac{r_{l+1}\Phi_{l}-r_l\Phi_{l+1}}{r_{l+1}-r_l}$ and $b_{2l}=\cfrac{\Phi_{l+1}-\Phi_l}{r_{l+1}-r_l}$. Defined $\hat{\Phi}_2=\displaystyle\max_{l=1,\cdots,N}\hat{\Phi}_{2l}$ is the upper approximation function.
\begin{theorem}
\label{thm: 5.2}
    Let $\hat{r_i}=(\hat{r}_{i1},\cdots, \hat{r}_{i,n})$. Together with the approximation of $\Phi^{-1}(p^{y_i^{1/\theta}})$, we have the following approximation of problem (P5), correspondingly:
    \begin{align*}
        (\text{P}7)~~\min & ~~q_1\Re(\mu_c^Hz)+q_2\sqrt{\frac{1}{2}\Re(z^H\Gamma_cz+z^HC_c\bar{z})}\\
        \text{s.t. }& \Re(\mu_{A_i}z) + \sqrt{\frac{1}{2}\Re(r_i^H\Gamma_{A_i}r_i+r_i^TC_{A_i}r_i)+\sigma_{b_i}}-\mu_{b_i}\leq0,\forall i\\
        & \Re(r_{ij})\geq \Re(a_{2l}z_j + b_{2l} m_{ij}),\quad\forall i,j,l\\
        & \Im(r_{ij})\geq \Im(a_{2l}z_j + b_{2l} m_{ij}),\quad\forall i,j,l\\
        & \sum_{i=1}^m m_{ij}=z_j,\forall j,\quad\Re(z)\geq0,\quad \Im(z)\geq 0,\quad m\geq 0
        \end{align*}
Moreover, if $p^{\sum y_i^{*\frac{1}{\theta}}}\leq\Phi\left(
        \frac{-\mu_{A_iz-b_i}}{\sigma_{A_iz-b_i}}\right)$, then the optimal value of this approximation is an upper bound of (P6).
\end{theorem}
\begin{proof}
    The same as proof Theorem \ref{thm: 5.1}.
\end{proof}
\section{Simulation Experiments} \label{sec: Section 8}

The implementation has been done within Python environment, our optimization problem was solved using CVXPY.
\subsubsection{Adaptive Beamforming with Mismatch Problem}
Our formalism has been applied for adaptive Beamforming using MVDR \cite{cox1987robust}. Assume that we have several sensors $N$, the output signal of the narrowband beamformer is $y(t)=w^Hx(t)$, where $t$ is the sample index,  $w\in\mathbb{C}^M$  beamformer weight coefficients, and $x(t)=S(t)+v(t)$ is the snapshot vector of array observations. where $S(t)=s(t)a_s$ and $v(t)$ are the desired signal and the interference-plus-noise components of $x(t)$, respectively, $s(t)$ and $a_s$ are the desired signal waveform and its steering vector (spatial signature). The optimal weight vector can be obtained by computing the maximum of the SINR function:
\begin{align}
    \text{SINR} = \frac{w^HR_sw}{w^HR_{i+n}w}\label{sinr}
\end{align}
$R_s, R_{i+n}$ are the signal's covariance and interference and noise's covariance. Hence, the maximization of (\ref{sinr}) is equivalent to \cite{kim2008robust}:
\begin{equation}
    \min w^HR_{i+n}w\quad\text{s.t. } \quad w^Ha_s=1
\end{equation}
where the optimal solution $w^*=R_{i+n}^{-1}a/a^HR_{i+n}^{-1}a$ with SINR$=\sigma_s^2a^HR_{i+n}^{-1}a$. In practice, the steering vector is estimated to have errors. The key idea of the beamformer developed in \cite{vorobyov2003robust} is to explicitly model the steering vector uncertainty as $\tilde{a}_s=\delta+a_s\neq a_s$ where $\tilde{a}_s$ and $a_s$ are the actual and presumed signal steering vectors, respectively. 
In \cite{vorobyov2008relationship}, Vorobyov et al. reformulated the problem as a probability-constrained problem and converted it to a deterministic one. In this work, we solved the probability-constrained problem using our formalism to optimize $w$. The CCCP problem is then:
\begin{align}
        \min & ~~w^H{R_{i+n}}w \label{eq 20}\\
        \text{s.t. } &\mathbb{P}[-\Re(\delta^Hw)\leq\Re(a^Hw) -1]\geq p, \quad \Re(a^Hw)\geq0,~~ \Im(a^Hw)=0 \label{eq 21}
\end{align}
with $\delta$ is circular i.i.d. zero-mean complex Gaussian. Using lemma (\ref{lemma 1}) and theorem (7.1), we have the following equivalent deterministic problem:
\begin{align}
        \min & ~~w^HR_{i+n}w \\
        \text{s.t. } & \Phi^{-1}(p)\|C_{\delta}^{1/2}w\|/2\leq \Re(a^Hw) -1,\quad \Re(a^Hw)\geq0,~~ \Im(a^Hw)=0
\end{align}
\begin{figure}[ht]
    \centering
    \includegraphics[width=0.9\linewidth]{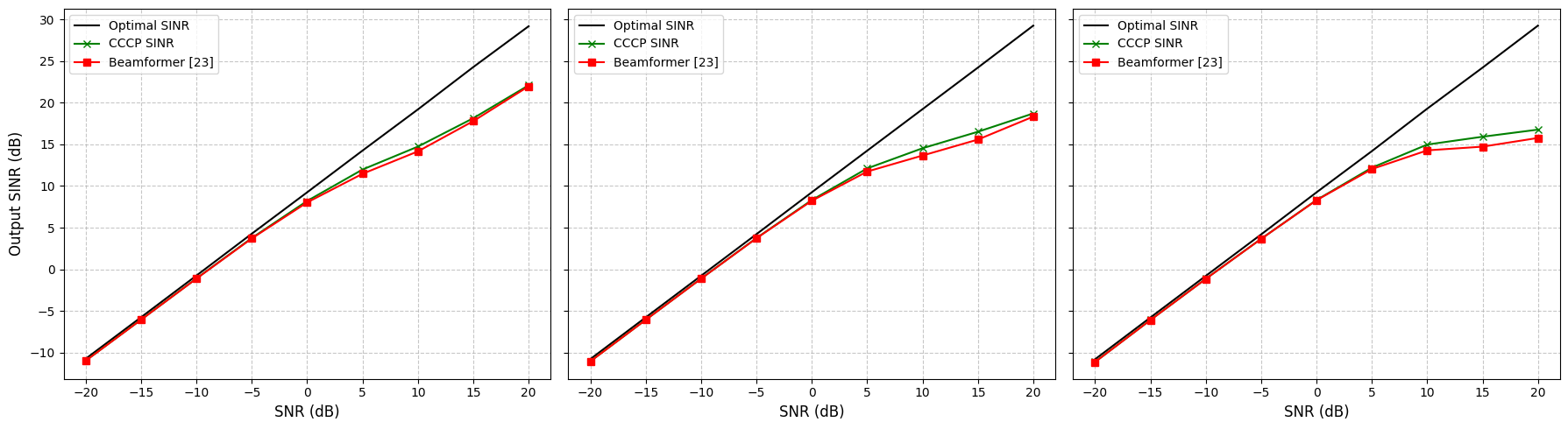}
    \caption{SINR versus SNR for INR = 5 (left), INR = 20 (middle), and INR = 40 (right).}    \label{fig: SINR vs SNR}
\end{figure}
The simulation parameters are configured as in \cite{vorobyov2008relationship}, the number of sensors \(M=8\), the number of samples \(K=100\), the spacing between array elements \(d=0.5\) wavelengths, and the probability \(p=0.95\). The Direction of Arrival (DOA) of the desired signal is \(\theta_s = 3^\circ\). Two interference signals have DOAs of \(30^\circ\) and \(50^\circ\). The perturbation variance is set as \(\sigma_\delta^2 = 0.3M\) with covariance \(Cov_\delta = \sigma_\delta^2 / M I_M\). The noise power is \(\sigma_n^2 = 1\), and 200 simulation runs.

Figure \ref{fig: SINR vs SNR} illustrates the output SINRs of our approach and the method presented in \cite{vorobyov2008relationship} as a function of the signal-to-noise ratio (SNR). The results demonstrate that our approach significantly outperforms the tested beamformers in \cite{vorobyov2008relationship}. Moreover, our approach solves the problem with any complex Gaussian distribution with independence between the real and imaginary parts. In contrast, prior works assume that the uncertainty in the steering vector follows an i.i.d. zero-mean Gaussian distribution. Furthermore, our approach accommodates multiple received signal scenarios.
\subsubsection{Comparisons between individual and joint CCCP}
Another constraint has been imposed on the problem (\ref{eq 20},\ref{eq 21}) which is $\Re(\delta_i w)+\Re(a_i^Hw)\leq \alpha$, with the same parameters as before, to suppress the interfering signal that might come from a strong source the problem was solved as Joint and Individual CCCP to compare between them, as expected and depicted in figure \ref{fig:enter-label}, joint CCCP outperforms the Individual one, as in the latter, you might sample from risk region.
\begin{figure}[ht]
    \centering
    \includegraphics[width=0.5\linewidth]{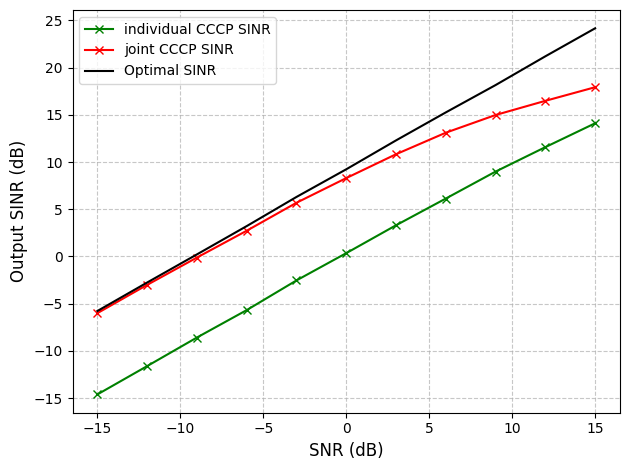}
    \caption{Output SINR with INR = $20$, $100$ runs and $\alpha = 0.7$. Problem was solved as an Individual CCCP and Joint CCCP.}
    \label{fig:enter-label}
\end{figure}

\section{Conclusion}\label{sec: Section 9}
This paper presents a novel approach to optimizing real-world applications involving complex variables through complex chance-constrained optimization programming (CCCP). Our work focuses on handling random experiments with specific confidence levels on constraints. We addressed a special case of complex optimization, linear programming with complex random variables following a complex normal distribution. We reformulated the problem to a deterministic one, resulting in a deterministic SOCP for individual CCCP. For joint CCCP, we formulated a deterministic problem involving the product of two inherently convex functions, using Taylor Series and Piecewise tangent approximations. The optimal values obtained serve as a tight gap between the lower and upper bounds for the original problem. Experimentation shows joint CCCP outperforms individual CCCP, providing enhanced accuracy and speed for many applications.


%
%
%
\bibliographystyle{splncs}

\end{document}